\documentclass[11pt, one side, article]{memoir}

\settrims{0pt}{0pt} % page and stock same size
\settypeblocksize{*}{35.1pc}{*} % {height}{width}{ratio}
\setlrmargins{*}{*}{1} % {spine}{edge}{ratio}
\setulmarginsandblock{.98in}{.98in}{*} % height of typeblock computed
\setheadfoot{\onelineskip}{2\onelineskip} % {headheight}{footskip}
\setheaderspaces{*}{1.5\onelineskip}{*} % {headdrop}{headsep}{ratio}
\checkandfixthelayout

\usepackage{amsthm}
\usepackage{mathtools}

\usepackage[inline]{enumitem}
\usepackage{ifthen}
\usepackage[utf8]{inputenc} %allows non-ascii in bib file
\usepackage{xcolor}

\usepackage[backend=biber, backref=true, maxbibnames = 10, style = alphabetic]{biblatex}
\usepackage[bookmarks=true, colorlinks=true, linkcolor=blue!50!black,
citecolor=orange!50!black, urlcolor=orange!50!black, pdfencoding=unicode]{hyperref}
\usepackage[capitalize]{cleveref}

\usepackage{tikz}

\usepackage{amssymb}
\usepackage{newpxtext}
\usepackage[varg,bigdelims]{newpxmath}
\usepackage{mathrsfs}
\usepackage{dutchcal}
\usepackage{mathalfa}
\usepackage{fontawesome}
\usepackage{ebproof}
\usepackage{stmaryrd}
\usepackage{ebproof}
\usepackage{graphicx}
\usepackage{setspace} % For \linespread and spacing

\usepackage{lipsum}

  \crefformat{enumi}{\card#2#1#3}
  \crefalias{chapter}{section}
  
\crefname{notation}{Notation}{Notations}
\Crefname{notation}{Notation}{Notations}

\crefname{remarkx}{Remark}{Remarks}
  \hypersetup{final}

  \setlist{nosep}
  \setlistdepth{6}

  \usetikzlibrary{ 
  	cd,
  	math,
  	decorations.markings,
		decorations.pathreplacing,
  	positioning,
  	arrows.meta,
  	shapes,
		shadows,
		shadings,
  	calc,
  	fit,
  	quotes,
  	intersections,
    circuits,
    circuits.ee.IEC
  }
  
  \tikzset{
biml/.tip={Glyph[glyph math command=triangleleft, glyph length=.95ex]},
bimr/.tip={Glyph[glyph math command=triangleright, glyph length=.95ex]},
}

\tikzset{
	tick/.style={postaction={
  	decorate,
    decoration={markings, mark=at position 0.5 with
    	{\draw[-] (0,.4ex) -- (0,-.4ex);}}}
  }
} 
\tikzset{
	slash/.style={postaction={
  	decorate,
    decoration={markings, mark=at position 0.5 with
    	{\draw[-] (.3ex,.3ex) -- (-.3ex,-.3ex);}}}
  }
}

\newcommand{\adj}[5][30pt]{%[size] Cat L, Left, Right, Cat R.
\begin{tikzcd}[ampersand replacement=\&, column sep=#1]
  #2\ar[r, shift left=6pt, "#3"]
  \ar[r, phantom, "\scriptstyle\Rightarrow"]\&
  #5\ar[l, shift left=6pt, "#4"]
\end{tikzcd}
}

\newcommand{\lkan}[6][Rightarrow]{
	\begin{tikzcd}[ampersand replacement=\&, row sep=23pt]
		#2\ar[r, "#3", ""' name=p]\ar[d, "#5"']\&#4\\
		#6\ar[ur, bend right=20pt, pos=.7, "\lens{#3}{#5}"', "" name=q]
		\ar[from=p, to=p|-q, xshift=-2pt, shorten >=-7pt, #1]
	\end{tikzcd}
}

\newcommand{\xtickar}[1]{\begin{tikzcd}[baseline=-0.5ex,cramped,sep=small,ampersand 
replacement=\&]{}\ar[r,tick, "{#1}"]\&{}\end{tikzcd}}

\theoremstyle{definition}
\newtheorem{definitionx}{Definition}[chapter]
\newtheorem{examplex}[definitionx]{Example}
\newtheorem{remarkx}[definitionx]{Remark}
\newtheorem{notation}[definitionx]{Notation and Terminology}

\theoremstyle{plain}

\newtheorem{theorem}[definitionx]{Theorem}
\newtheorem{proposition}[definitionx]{Proposition}
\newtheorem{corollary}[definitionx]{Corollary}
\newtheorem{lemma}[definitionx]{Lemma}

\newtheorem*{theorem*}{Theorem}
\newtheorem*{proposition*}{Proposition}
\newtheorem*{corollary*}{Corollary}
\newtheorem*{lemma*}{Lemma}
\newtheorem*{warning*}{Warning}

\newenvironment{example}
  {\pushQED{\qed}\examplex}
  {\popQED\endexamplex}
  
 \newenvironment{remark}
  {\pushQED{\qed}\remarkx}
  {\popQED\endremarkx}
  
  \newenvironment{definition}
  {\pushQED{\qed}\definitionx}
  {\popQED\enddefinitionx}

\DeclareSymbolFont{stmry}{U}{stmry}{m}{n}
\DeclareMathSymbol\fatsemi\mathop{stmry}{"23}

\DeclareFontFamily{U}{mathx}{\hyphenchar\font45}
\DeclareFontShape{U}{mathx}{m}{n}{
      <5> <6> <7> <8> <9> <10>
      <10.95> <12> <14.4> <17.28> <20.74> <24.88>
      mathx10
      }{}
\DeclareSymbolFont{mathx}{U}{mathx}{m}{n}
\DeclareFontSubstitution{U}{mathx}{m}{n}
\DeclareMathAccent{\widecheck}{0}{mathx}{"71}

\ExplSyntaxOn
\NewDocumentEnvironment{sequation}{O{\fontsize{15pt}{15pt}\selectfont
}b}
 {
  \yufip_sequation:nnn {equation}{#1}{#2}
 }{}
\NewDocumentEnvironment{sequation*}{O{\fontsize{16pt}{16pt}\selectfont
}b}
 {
  \yufip_sequation:nnn {equation*}{#1}{#2}
 }{}
\cs_new_protected:Nn \yufip_sequation:nnn
 {
  \begin{#1}
  \mbox{#2$\displaystyle#3$}
  \end{#1}
 }
\ExplSyntaxOff

\renewcommand{\ss}{\subseteq}

\DeclareMathOperator{\Hom}{Hom}
\DeclareMathOperator{\Mor}{Mor}
\DeclareMathOperator{\dom}{dom}
\DeclareMathOperator{\cod}{cod}

\DeclareMathOperator{\ob}{Ob}

\DeclareMathOperator{\el}{El}

\newcommand{\cat}[1]{\mathcal{#1}}%a generic category
\newcommand{\Cat}[1]{\mathbf{#1}}%a named category
\newcommand{\Fun}[1]{\textit{#1}}%a named functor

\newcommand{\hh}[2][]{#1 \tn{#2} #1}
\newcommand{\qqand}{\hh[\qquad]{and}}
\newcommand{\qand}{\hh[\quad]{and}}

\newcommand{\hi}[4][]{#1 #2 \tn{#4} #3}
\newcommand{\where}[1][,]{\hi[#1]{\qquad}{\quad}{where}}

\newcommand{\id}{\mathrm{id}}
\newcommand{\then}{\mathbin{\fatsemi}}

\newcommand{\too}{\longrightarrow}
\newcommand{\mapstoo}{\longmapsto}

\newcommand{\To}[2][]{\xrightarrow[#1]{\tn{$#2$}}}

\newcommand{\from}{\leftarrow}
\newcommand{\fromm}{\longleftarrow}

\newcommand{\From}[1]{\xleftarrow{#1}}

\newcommand{\imp}{\Rightarrow}
\newcommand{\limp}{\Leftarrow}
\newcommand{\Imp}[2][]{\xRightarrow[#1]{\tn{$#2$}}}

\newcommand{\tickar}{\xtickar{}}

\newcommand{\card}{\,^{\#}}

\newcommand{\op}{^\tn{op}}
\newcommand{\co}{^\tn{co}}

\newcommand{\tn}[1]{\textnormal{#1}}

\newcommand{\ul}[1]{\underline{#1}}

\newcommand{\nn}{\mathbb{N}}

\newcommand{\bo}{\textit{bo}}
\newcommand{\ff}{\textit{ff}}
\newcommand{\dopf}{\textit{dopf}}

\newcommand{\smset}{\Cat{Set}}
\newcommand{\smcat}{\Cat{Cat}}
\newcommand{\ssmcat}{\mathbb{C}\Cat{at}}
\newcommand{\ssmcata}{\mathbb{C}\Cat{at}_{\mathbf{A}}}
\newcommand{\pprof}{\mathbb{P}\Cat{rof}}
\newcommand{\pprofcat}[1][]{\mathbb{P}\Cat{rof}_{#1\,}\tn{-}\smcat}
\newcommand{\lladj}{\mathbb{L}\Cat{Adj}}
\newcommand{\End}{\Cat{End}}
\newcommand{\mnd}{\Cat{Mnd}}
\newcommand{\comnd}{\Cat{Cmd}}

\newcommand{\List}{\Fun{list}}

\newcommand{\set}{\tn{-}\Cat{Set}}
\newcommand{\paths}{\textit{paths}}
\newcommand{\frcat}{\textit{frcat}}
\newcommand{\simp}{\mathbf{\Delta}}

\newcommand{\yon}{{\mathcal{y}}}
\newcommand{\poly}{\Cat{Poly}}

\newcommand{\cofree}{\mathfrak{c}}
\newcommand{\free}{\mathfrak{m}}

\newcommand{\uu}{u}

\newcommand{\biglens}[3][\vphantom{f_f^f}]{
     \begin{bmatrix}{#1#3} \\ {#1#2} \end{bmatrix}
}
\newcommand{\littlelens}[3][\vphantom{f}]{
     \begin{bsmallmatrix}{#1#3} \\ {#1#2} \end{bsmallmatrix}
}
\newcommand{\lens}[2]{
  \relax\if@display
     \biglens{#1}{#2}
  \else
     \littlelens{#1}{#2}
  \fi
}
\newcommand{\lc}[2]{
  \relax\if@display
     \biglens[]{\!\!#1\!\!}{\!\!#2\!\!}
  \else
     \littlelens[]{\!\!#1\!\!}{\!\!#2\!\!}
  \fi
}

\newcommand{\catcmd}[1]{k_{\cat{#1}}}
\newcommand{\elts}[1][]{\textit{elts}_{#1}}

\newcommand{\bul}[1]{\overset{#1}{\bullet}}

\newcommand{\removethis}[1]

\makeatletter
\newcommand{\ifpagenotequal}[4]{%
    \edef\pageX{\pageref{#1}}%
    \edef\pageY{\pageref{#2}}%
    \ifx\pageX\pageY
        #3%
    \else
        #4%
    \fi
}
\makeatother

\newcommand{\keywords}[1]{%
  \par\noindent\textbf{Keywords: }#1\par
}
\newcommand{\thanksAFOSR}[1]{This material is based upon work supported by the Air Force Office of Scientific Research under award number #1.}

\setsecnumdepth{section}
\settocdepth{section}
\begin{document}

\title{Categories by Kan extension}

\author{David I. Spivak}

\date{\vspace{-.1in}}

\begingroup
\linespread{1}\selectfont
\maketitle

\begin{abstract}
Categories can be identified---up to isomorphism---with polynomial comonads on $\smset$. The left Kan extension of a functor along itself is always a comonad---called the density comonad---so it defines a category when its carrier is polynomial. We provide a number of generalizations of this to produce new categories from old, as well as from distributive laws of monads over comonads. For example, all Lawvere theories, all product completions of small categories, and the simplicial indexing category $\simp\op$ arise in this way. Another, seemingly much less well-known, example constructs a so-called \emph{selection} category from a polynomial comonad in a way that's somehow dual to the construction of a \emph{Lawvere theory} category from a monad; we'll discuss this in more detail. Along the way, we will see various constructions of non-polynomial comonads as well.\\
\end{abstract}
\endgroup

%\tableofcontents*

\chapter{Introduction}\label{chap.intro}

In this paper we provide new methods for constructing categories from other data, using left Kan extensions \cite{kan1958adjoint,macLane1998categories}. Of course, left Kan extensions (of functors along functors) are \emph{functors}; in what sense could they also be \emph{categories}? A recent result of Ahman and Uustalu \cite{ahman2016directed} allows us to make the connection by showing that categories $\cat{C}$ (up to isomorphism) can be identified with comonads on $\smset$ for which the underlying functor $\catcmd{C}\colon\smset\to\smset$ is polynomial (preserves all connected limits).%
\footnote{
In a \href{https://mathoverflow.net/questions/457580/examples-of-non-polynomial-comonads-on-set}{mathoverflow post} Simon Henry and Kevin Carlson explained that comonads whose underlying functor $\smset\to\smset$ preserves \emph{finite} connected limits can be identified with \emph{ionads}: Grothendieck toposes equipped with a separating set of points \cite{garner2011ionads}. Aaron David Fairbanks gave examples of still more general comonads on $\smset$. \cref{chap.Kan,chap.density_comonad} give constructions for such comonads.
}
The purpose of this paper is to give a variety of ways to produce comonads using Kan extensions.

We denote a left Kan extension of $p$ along $q$ by $\lens{p}{q}\coloneqq\mathrm{Lan}_q(p)$.%
\footnote{This paper uses a large line-spread in order to typographically accommodate the $\lens{-}{-}$ notation.} 
It is well-known that the left Kan extension $\lens{p}{p}$ of a functor $p\colon\cat{C}\to\cat{D}$ along itself, when it exists, always carries a comonad structure. We will generalize this by showing that if $\cat{C}$ and $\cat{D}$ are accessible categories, e.g.\ copresheaf categories, then for any accessible monad $t$, accessible comonad $k$, and distributive law $t\circ k\to k\circ t$, the left Kan extension $\lens{p\circ k}{p\circ t}$ exists and has a comonad structure. We will see in \cref{thm.monad_comonad_dist} that---with no further assumptions on $t$---when $p$ and $k$ are polynomial functors, this left Kan extension comonad, so it can be identified with a category.

We offer several examples. It is not too surprising that every category $\cat{C}$ arises in this way---as we will see in \cref{ex.recover_site}---essentially because the functor $U_\cat{C}\colon\cat{C}\set\to\smset$ given by $U_\cat{C}(X)\coloneq\sum_{c:\ob\cat{C}}X(c))$ is comonadic
\[
\begin{tikzcd}
	\cat{C}\set\ar[r, "U_\cat{C}", ""' name=top]\ar[d, "U_\cat{C}"']&\smset\\
	\smset\ar[ur, bend right=25pt, "\catcmd{C}"', "" name=bot]
	\ar[from=top, to=bot-|top, shorten=1pt, Rightarrow]
\end{tikzcd}
\]
In other words, $\cat{C}$ corresponds to the ``density'' comonad associated to $U_\cat{C}$. But it is more interesting to ``get more out than you put in,'' i.e.\ to find interesting categories arising from much less data. For example, we will see that 
\begin{itemize}
	\item the simplicial indexing category, $\simp\op$ arises as the left Kan extension $\lens{\paths}{\paths\;\circ\;\frcat}$ of one functor from the category of graphs to $\smset$ along another (\cref{prop.deltaop}), 
	\item the walking arrow \fbox{$\bullet\to\bullet$} arises as the left Kan extension $\lens{\yon+1}{\yon+1}$ of the maybe monad along itself (\cref{ex.walking}), and
	\item the free product completion of a category $\cat{C}$ arises as the left Kan extension $\lens{\List\circ\catcmd{C}}{\List}$ (\cref{prop.prod_completion}). In particular $\lens{\List}{\List}$ corresponds to the category $\Cat{FinSet}\op$.
\end{itemize}

It is well-known that $S$-sorted finitary Lawvere theories $\cat{T}$ \cite{lawvere1963functorial} have associated finitary monads $t$ on $S\set$. Namely, the finite product category associated to $t$ is the full subcategory of the opposite of the Kleisli category $S\set_t$ spanned by the finite $S$-sets. The previous sentence is a mouthful, but we can recover this category more concisely again using a generalization of the density comonad: $\cat{T}\cong\lens{\List}{\List\,\circ\, t}$; see \cref{ex.lawvere}. 

Categories of the form $\lens{p\circ\catcmd{C}}{p}$ seem to be new, other than being discussed in a \href{https://topos.institute/blog/2021-12-30-selection-categories/}{blog post} by the author. There they were called \emph{selection categories}, referring to their proposed use in modeling natural selection. They arise from \cref{thm.colax}, which says that in a fairly general setting, the functor $\lens{p\circ-}{p}$ is colax monoidal:
\[
  \lens{p}{p}\imp\id_\cat{C}.
  \qqand
  \lens{p\circ q\circ r}{p}\imp\lens{p\circ q}{p}\circ\lens{p\circ r}{p}.
\]
In \cref{thm.sel_enriched,prop.selection_boff}, we prove that the selection category construction $\lens{p\circ-}{p}$ comprises a functor $\smcat\to\pprofcat$, from the category of small categories to that of $\pprof$-enriched categories, and that this functor respects the bijective-on-objects, fully-faithful factorization system.

\begin{notation}\label{not.main}

We may denote composition $A\To{f}B\To{g}C$ by $f\then g$ or by $g\circ f$. 
We denote the identity map (e.g.\ on an object $A$) 
%using the name of the object $A\To{A}A$ or more verbosely 
by $A\To{\id_A}A$. 
We use $1$ to denote a terminal object in any category that has one. 

Given a set $I$ and a $I$-shaped diagram $A\colon I\to\cat{S}$, we denote the coproduct---when it exists---by $\sum_{i:I}A_i$. As here, we sometimes write $A_i$ in place of $A(i)$.

For any small category $\cat{C}$ and object $c:\ob\cat{C}$, let
\[\cat{C}[c]\coloneqq\sum_{c':\ob\cat{C}}\cat{C}(c,c')\]
denote the set of all $\cat{C}$-morphisms emanating from $c$. If $I\colon\cat{C}\to\smset$ is any functor, we denote its category of elements by $\el_\cat{C}I$. We denote by $\cat{C}\set$ the category of $\cat{C}$-copresheaves, i.e.\ of functors $\cat{C}\to\smset$ and natural transformations between them.

%Recall that a category is called \emph{cocomplete} if it has all small colimits.

We denote the 1-category of small categories and functors by $\smcat$; we denote by $\ssmcat$ the 2-category that also includes natural transformations. We denote by $\lladj\ss\ssmcat$ its wide and locally-full subcategory consisting of small categories, \emph{left adjoint} functors, and natural transformations between left adjoints. 
%We denote adjunctions $L\dashv R$ as follows
%\[
%\adj{\cat{C}}{L}{R}{\cat{D}}
%\qqor
%\adjr{\cat{D}}{R}{L}{\cat{C}}
%\]
%where the 2-cell points in the same direction as the \emph{left adjoint}. The reason is that it indicates both the unit $\eta\colon\id_\cat{C}\imp L\then R$ and the counit $R\then L\imp\id_\cat{D}$.

We denote by $\ssmcata\ss\ssmcat$ the locally-full subcategory consisting of accessible categories, accessible functors, and all natural transformations. Whenever we speak of an accessible functor, we will implicitly assume it is between accessible categories. Given an accessible category $\cat{C}$, we write $\End(\cat{C})$ for the category of accessible endofunctors $\cat{C}\to\cat{C}$; it comes equipped with a monoidal structure $(\id_\cat{C},\circ)$. We denote the categories of accessible monads and comonads on $\cat{C}$ by $\mnd(\cat{C})$ and $\comnd(\cat{C})$ respectively.

Given a set $A$, we denote the endofunctor it represents by
\[\yon^A\coloneqq\smset(A,-)\colon\smset\to\smset.\]
We may write $1$ for $\yon^0$ and $\yon$ for $\yon^1$; note that $1$ is terminal and that $\yon$ is the identity functor and the unit for the $(\circ$)-monoidal product on $\End(\smset)$.

When a map is unique up to unique isomorphism, we may refer to it using the definite article ``the'', rather than ``a''.

\end{notation}
\section*{Acknowledgments}
\thanksAFOSR{FA9550-23-1-0376}

\chapter{Left Kan extensions}\label{chap.Kan}

This section is probably mostly review for experts, though we have not found all the material written elsewhere; see \cite{kelly1982basic}, \cite[Appendix A]{diliberti2020codensity}.

Let $q\colon \cat{C}\to\cat{C'}$ be an accessible functor between accessible categories,%
\footnote{For readers who are not very familiar with accessible categories, the term is basically a smallness condition. Every small category is Morita equivalent to an accessible one (namely its idempotent completion) \cite{makkai1989accessible}, and every Grothendieck topos (e.g.\ $\smset$) is accessible. Accessibility is the ``ability to access every object by small ones,'' in a certain sense.}
and let $\cat{D}$ be a cocomplete accessible category. In that case, the \emph{left Kan extension along $q$}, which we will denote $\lens{-}{q}\coloneqq\mathrm{Lan}_q(-)$, exists \cite{adamek1994locally} and is the left adjoint to $q$-precomposition:%
\footnote{The 2-cell in \eqref{eqn.global_Kan} points in both the direction of the unit and the counit of the adjunction $\lens{-}{q}\dashv (-\circ q)$.} 
	\begin{equation}\label{eqn.global_Kan}
	\adj{\ssmcat(\cat{C},\cat{D})}{\lens{-}{q}}{-\circ q}{\ssmcat(\cat{C'},\cat{D})}.
	\end{equation}
In other words, for any $p\colon\cat{C}\to\cat{D}$, there is a natural bijection
\begin{equation}\label{eqn.coclosure}
	\Hom\left(\lens{p}{q},r\right)\cong\Hom(p,r\circ q).
\end{equation}
The left Kan extension of $p$ along $q$ is unique up to unique isomorphism and is displayed diagonally here:
	\[
	\lkan{\cat{C}}{p}{\cat{D}}{q}{\cat{C'}}
	\]
As indicated, it comes equipped with a natural transformation, which we denote by
\begin{equation}\label{eqn.coeval}
	\eta_p^q\colon p\imp \lens{p}{q}\circ q,
\end{equation}
and which is given by the unit of the defining adjunction \eqref{eqn.global_Kan}. Given any functor $p'\colon\cat{C'}\to\cat{D}$, the counit of \eqref{eqn.global_Kan} also provides a natural transformation
\begin{equation}\label{eqn.counit}
	\epsilon_{p'}^q\colon\lens{p'\circ q}{q}\imp p'.
\end{equation}
The adjunction's triangle equations say that the following commute
\begin{equation}\label{eqn.triangle}
\begin{tikzcd}
	r\circ q\ar[r, "\eta_{r\circ q}^q"]\ar[rr, bend right, equals]&
	\lens{r\circ q}{q}\circ q\ar[r, "\epsilon_r^q\circ q"]&
	r\circ q
\end{tikzcd}
\qand
\begin{tikzcd}[column sep=30pt]
	\lens{p}{q}\ar[r, "\lens{\eta_{p}^q}{q}"]\ar[rr, bend right, equals]&
	\lens{\lens{p}{q}\circ q}{q}\ar[r, "\epsilon_{\littlelens{p}{q}}^q"]&
	\lens{p}{q}
\end{tikzcd}
\end{equation}

\begin{remark}
In full generality, the right Kan extension of $p\colon\cat{C}\to\cat{D}$ along $q\colon\cat{C}\to\cat{C'}$ is given by the left Kan extension $\lens{p\op}{q\op}$ of $p\op$ along $q\op$. Right Kan extensions will play very little role in this paper.
\end{remark}

\begin{proposition}\label{prop.functoriality}
The symbol $\lens{p}{q}$ is covariantly functorial in $p\colon\cat{C}\to\cat{D}$ and contravariantly functorial in $q\colon\cat{C}\to\cat{C'}$. That is, given $\varphi\colon p\imp p'$ and $\psi\colon q'\imp q$ we have
\[
	\lens{\varphi}{\psi}\colon\lens{p}{q}\imp\lens{p'}{q'}.
\]
\end{proposition}
\begin{proof}
For the covariant functoriality of $\lens{-}{q}\colon\ssmcat(\cat{C},\cat{D})\to\ssmcat(\cat{C'},\cat{D})$, see \eqref{eqn.global_Kan}. For the contravariant functoriality of $\lens{p}{-}\colon\ssmcat(\cat{C},\cat{C'})\op\to\ssmcat(\cat{C'},\cat{D})$, suppose given a map $\psi\colon q'\imp q$ in the category $\ssmcat(\cat{C},\cat{C}')$, and consider the composite
\[
p\Imp{\eta_p^{q'}}\lens{p}{q'}\circ q'\Imp{\lens{p}{q'}\circ\psi}\lens{p}{q'}\circ q
\]
By \eqref{eqn.coclosure}, this induces a map $\lens{p}{q}\imp\lens{p}{q'}$. The preservation of identities and composites is easy to check.
\end{proof}

\begin{example}
	Let $P,Q:\smset$ be sets. The left Kan extension of $P\colon 1\to\smset$ along $Q\colon 1\to\smset$ is the functor $\smset\to\smset$ shown here:
	\[
  	\lkan{1}{P}{\smset}{Q}{\smset}
	\where
	\lens{P}{Q}\colon X\mapsto P\times X^Q
	\]
	Indeed, there is a map $P\to P\times Q^Q$ given by $\mathrm{p}\mapsto(\mathrm{p},\id_Q)$, and for any other $F\colon\smset\to\smset$ equipped with $f\colon P\to F(Q)$, there is a natural transformation
	\[
		\lens{P}{Q}\To{\lens{f}{Q}}\lens{F\circ Q}{Q}\To{\epsilon^Q_F}F.
	\]
	Explicitly, in this case its component at $X:\smset$ is
	\[
  	P\times X^Q\To{f\times X^Q}
  	F(Q)\times X^Q\To{\epsilon_F^Q(X)}F(X)
	\]
	where the counit $\epsilon_F^Q(X)$ is the composite $F(Q)\times X^Q\to F(Q)\times F(X)^{F(Q)}\to F(X)$: apply $F$ to the hom-set and then evaluate. We leave the remaining details to the reader.
\end{example}

\begin{remark}\label{rem.closed_analogy}
Consider the case that $\cat{C}=\cat{C'}=\cat{D}$. The category $\End(\cat{C})$ of accessible endofunctors on $\cat{C}$ has a monoidal structure $(\id_\cat{C},\circ)$ given by composition of functors, and \eqref{eqn.coclosure} exhibits $\lens{p}{q}$ as a \emph{coclosure} for this structure. 

Thinking of Kan extension as analogous to a (non-symmetric) monoidal closure, e.g.\
\[\Hom(Q\otimes P',P)\cong\Hom(P',[Q,P]),\]
can provide intuition for thinking about Kan extensions, whose notation $\lens{p}{q}$ is analogous to $[Q,P]$.%
\footnote{The notation $\lens{p}{q}$ also comes from notation for lenses; see \cite{hedges2018limits,myers2020double,spivak2019generalized}.} Beyond having the same sort of functoriality and unit and counits
\begin{align*}
  &p\to\lens{p}{q}\circ q&&
  \lens{p'\circ q}{q}\to p'\\
  &Q\otimes [Q,P]\to P
  &&
  P'\to[Q,P'\otimes Q],
\end{align*}
there are many other helpful analogies as well, e.g.\ a $[-,-]$ analogue to \cref{eqn.along_composites}.
\end{remark}

Note that left adjoint functors between accessible categories are accessible.

\begin{theorem}\label{thm.catco_to_ladj}
Kan extension into a cocomplete category $\cat{D}$ constitutes a 2-functor of the form
\[
\ssmcat(-,\cat{D})\colon\ssmcata\co\to\lladj
\]
sending $\cat{C}\mapsto\ssmcat(\cat{C},\cat{D})$ and $q\colon\cat{C}\to\cat{C'}$ to the left adjoint $\lens{-}{q}\colon\ssmcat(\cat{C},\cat{D})\to\ssmcat(\cat{C'},\cat{D})$.

In particular, for accessible functors $\cat{C}\To{q}\cat{C'}\To{q'}\cat{C''}$ and a functor $p\colon\cat{C}\to\cat{D}$, there are natural isomorphisms:
\begin{equation}\label{eqn.along_composites}
\lens{\lens{p}{q}}{q'}\cong\lens{p}{q'\circ q}
\qqand
p\cong\lens{p}{\id_\cat{C}}.
\end{equation}
\end{theorem}
\begin{proof}
Given an accessible functor $q\colon\cat{C}\to\cat{C'}$, we have a left adjoint $\lens{-}{q}$ as in \eqref{eqn.global_Kan}. And given a map $q\limp q'$ in $\ssmcat\co(\cat{C},\cat{C'})$, we have a natural transformation $\lens{-}{q}\imp\lens{-}{q'}$ between the associated left adjoints, by \cref{prop.functoriality}. 

We now turn to composition. The right-hand isomorphism in \eqref{eqn.along_composites} follows directly from \eqref{eqn.coclosure}. At issue for the left-hand isomorphism are setups like this:
\begin{equation}\label{eqn.along_composites_diag}
\begin{tikzcd}[row sep=small]
\cat{C}\ar[r, "p"]\ar[d, "q"']&\cat{D}\\
\cat{C'}\ar[d, "q'"']\\
\cat{C''}
\end{tikzcd}
\end{equation}
The maps in either direction are induced by \eqref{eqn.coclosure} and \eqref{eqn.coeval}
\[
p\imp\lens{p}{q}\circ q\imp\lens{\lens{p}{q}}{q'}\circ q'\circ q
\qqand
p\imp\lens{p}{q'\circ q}\circ q'\circ q
\]
and we leave it to the reader to show (using \eqref{eqn.triangle}) that they are mutually inverse and to provide the remaining 2-cell composition data and coherences.
\end{proof}

%\begin{remark}
%There is \emph{not} a statement analogous to \eqref{eqn.along_composites} when the setup \eqref{eqn.along_composites_diag} is replaced by the setup
%\[
%\begin{tikzcd}
%  \cat{C}\ar[r, "p"]\ar[d, "q"']&\cat{D}\ar[r, "p'"]&\cat{D'}\\
%  \cat{C'}
%\end{tikzcd}
%\]
%%where $\cat{D'}$ is again a cocomplete category and $p'\colon\cat{D}\to\cat{D'}$ is a (not necessarily colimit-preserving) functor. 
%However, there is a natural map
%\[
%\lens{p'\circ p}{q}\imp p'\circ\lens{p}{q}
%\]
%and for $p''\colon\cat{D'}\to\cat{D''}$ to another cocomplete category, the following diagram commutes
%\[
%\begin{tikzcd}
%	\lens{p_2\circ p_1\circ p}{q}\ar[r, Rightarrow]\ar[rd, Rightarrow]&p_2\circ \lens{p_1\circ p}{q}\ar[d, Rightarrow]\\
%	&p_2\circ p_1\circ\lens{p}{q}
%\end{tikzcd}
%\]
%We will not need this again.
%\end{remark}

\begin{proposition}\label{prop.right_adjoint_kan}
If $\ell\colon\cat{C}\to\cat{D}$ has a right adjoint, then we can obtain it by left Kan extension,
\[
	\lkan{\cat{C}}{\id_\cat{C}}{\cat{C}}{\ell}{\cat{D}}
\]
\end{proposition}
\begin{proof}
Let $r\colon\cat{D}\to\cat{C}$ be right adjoint to $\ell$. The adjunction unit $\id_\cat{C}\to r\circ\ell$ induces a map $\lens{\id_\cat{C}}{\ell}\to r$. We obtain $r\to\lens{\id_\cat{C}}{\ell}$ by
\[
\begin{tikzcd}
	\cat{D}\ar[r, "r"]\ar[drr, bend right=20pt, equal, "" name=idD]&
	|[alias=C]|\cat{C}\ar[rr, equal, ""' name=idC]\ar[dr, bend right=10pt, "\ell" description]&&
	\cat{C}
	\\&&
	|[alias=D]|\cat{D}\ar[ru, bend right=10pt, pos=.6, "\lens{\id_\cat{C}}{\ell}"']
	\ar[from=C, to=idD-|C, Rightarrow]
	\ar[from=idC-|D, to=D, Rightarrow, shorten <= 2pt]
\end{tikzcd}
\]
It is easy to check that these two maps are mutually inverse, giving the required isomorphism $\lens{\id_\cat{C}}{\ell}\cong r$.
\end{proof}

%When it exists, the colimit of a diagram $X\colon I\to\cat{C}$ is well known to be computable as its left Kan extension $\colim(X)\cong\lens{X}{!}$ along $!\colon I\to 1$; by \cref{prop.right_adjoint_kan}, we see that the \emph{limit} of $X$ is also computable using the left Kan extension $\lim(X)\cong\lens{\id_{\cat{C}}}{\Delta}(X)$ of the identity along the diagonal $\Delta\colon\cat{C}\to\cat{C}^I$.
%
\chapter{Pra-functors between copresheaf categories} 

Calculations of left Kan extensions are particularly easy when the categories are copresheaf categories and the functors are parametric right adjoint (pra-) functors. We begin by reminding the reader what a pra-functor between copresheaf categories is.

\begin{lemma}\label{lemma.pra}
For any functor $p\colon\cat{C}\set\to\cat{D}\set$, the following conditions are equivalent:
\begin{enumerate}[label=(\alph*)]
	\item Naturally in $X:\cat{C}\set$ and $d:\cat{D}$, the functor $p$ has the form
	\[
		p(X)(d)\cong\sum_{i:I_d}\cat{C}\set\big(A(d,i),X\big),
	\]
where $I\colon\cat{D}\to\smset$ and $A\colon(\el_\cat{D} I)\op\to\cat{C}\set$ are functors.
	\item the functor $p'\colon\cat{C}\set\cong\cat{C}\set/1\to\cat{D}\set/p(1)$, induced by slicing $p$ over the terminal object $1:\cat{C}\set$, is a right adjoint.%
	\footnote{This is the usual definition of pra-functor, and it generalizes beyond copresheaf categories. If $\Cat{C}$ has a terminal object $1$, then a functor $p\colon\Cat{C}\to\Cat{D}$ is a \emph{pra-functor} iff the induced functor $p'\colon\Cat{C}\cong\Cat{C}/1\to\Cat{D}/p(1)$ is a right adjoint.}
	\end{enumerate}
\end{lemma}
\begin{proof}
We prove each implication in turn. 

$(a)\imp (b)$: Assuming $(a)$, note that $\cat{C}\set(-,1_\cat{C})\cong 1$, so $p(1)\cong I$, and we need to prove that $\cat{C}\set\To{p'}\cat{D}\set/I$ is a right adjoint. But there is an isomorphism of categories $\cat{D}\set/I\cong(\el_\cat{D}I)\set$, and for any $(d,i):\ob(\el_\cat{D}I)$, we have $p'(X)(d,i)\cong\cat{C}\set(A(d,i),X)$. It follows that $p'\dashv\lens{A}{\yon^-}$ is right adjoint to the Yoneda extension, shown diagonally in the following commutative diagram:
\[
\lkan[equal]{(\el_\cat{D}I)\op}{A}{\cat{C}\set}{\yon^-}{(\el_\cat{D}I)\set}
\]
%Indeed,
%\begin{align*}
%  \cat{C}\set\left(\lens{A}{\yon^-}(Y),X\right)&\cong
%  \cat{C}\set\left(\colim_{\yon^{(d,i)}\to Y}A(d,i), X\right)\\&\cong
%	\lim_{\yon^{(d,i)}\to Y}p'(X)(d,i)\\&\cong
%	\lim_{\yon^{(d,i)}\to Y}(\el_\cat{D}I)\set(\yon^{(d,i)},p'(X))\\&\cong
%	(\el_\cat{D}I)\set(Y,p'(X)).
%\end{align*}

$(b)\imp (a)$: Assuming (b), let $I\coloneqq p(1)$ and let $\cat{D}\set/I\cong(\el_\cat{D}I)\set\To{A'}\cat{C}\set$ be the left adjoint of $p'$. Applying it to representable $(\el_\cat{D}I)$-sets defines $A$ and the required natural isomorphism.
\end{proof}

\begin{definition}[Pra-functor]
We say that $p\colon\cat{C}\set\to\cat{D}\set$ is a \emph{pra-functor} when either (and hence both) of the equivalent conditions in \cref{lemma.pra} hold. A morphism $p\imp q$ of pra-functors is just a natural transformation between the underlying functors.
\end{definition}

\begin{remark}[Profunctors are pra-functors]\label{rem.pro_is_pra}
A profunctor $\cat{D}\xtickar{A}\cat{C}$ can be identified with a pra-functor $\cat{C}\set\to\cat{D}\set$ for which the $I\colon\cat{D}\to\smset$ in \cref{lemma.pra}(a) is terminal, $I(-)=1$, since then $\el_\cat{D} I\cong\cat{D}$. Note, however, that maps between them are opposite.
\end{remark}

\begin{lemma}\label{lemma.pra_accessible}
Copresheaf categories are accessible, and pra-functors between them are accessible.
\end{lemma}
\begin{proof}
Since every set is a filtered colimit of finite sets, and every map from a finite set to a directed colimit of sets factors through a finite stage, $\smset$ is accessible. By \cite{nlab:accessible_category}, the category of functors from a small category to an accessible category is accessible, so copresheaf categories are accessible. Again by \cite{nlab:accessible_category}, any left or right adjoint functor between accessible categories is accessible, and by \cref{lemma.pra}(b), a pra-functor $p\colon\cat{C}\set\to\cat{D}\set$ between copresheaf categories can be written as a right adjoint $p'\colon\cat{C}\set\to\cat{D}\set/p(1)$ followed by a left adjoint $\cat{D}\set/p(1)\to\cat{D}\set$.
\end{proof}

In the case $\cat{C}=1=\cat{D}$, one may identify a pra-functor $p\colon\smset\to\smset$ with a polynomial functor, i.e.\ a coproduct of representables. Indeed, with $I:\smset$ and $A\colon I\to\smset$ as in \cref{lemma.pra}(a), then with notation \cref{not.main} we can write 
\[p\cong\sum_{i:I}\smset(A_i,-)=\sum_{i:I}\yon^{A_i}\]
as a coproduct of representables; the converse is easy too.

\begin{proposition}\label{prop.pra_simple}
Let $\cat{C}$ be a small category and let $q\colon\cat{C}\set\to\smset$ be an accessible functor. Given a pra-functor $p\colon\cat{C}\set\to\smset$, the left Kan extension of $p$ along $q$
\[
\lkan{\cat{C}\set}{p}{\smset}{q}{\smset}
\]
is also a pra-functor (i.e.\ a polynomial functor). Namely, if $p\cong\sum_{i:I}\cat{C}\set(A_i,-)$ for some $A\colon I\to\cat{C}\set$, then
\begin{equation}\label{eqn.kan_pra_to_set}
\lens{p}{q}\cong\sum_{i:I}\yon^{q(A_i)}.
\end{equation}
\end{proposition}
\begin{proof}
It suffices to provide a bijection, natural in $p$ and $r\colon\smset\to\smset$, between the following sets of natural transformations
\[
  \Hom\left(\sum_{i:I}\smset(q(A_i),-),r\right)
  \cong^?
	\Hom\left(\sum_{i:I}\cat{C}\set(A_i,-),r\circ q\right).
\]
It arises directly from the universal property of coproducts and the Yoneda lemma.
\end{proof}

\begin{remark}\label{rem.pra_general}
In fact, \cref{prop.pra_simple} generalizes significantly, and the added proof difficulty is only bookkeeping; we leave it as an exercise to the reader (or see \cite{spivak2025functorial}). 

Given small categories $\cat{C}$, $\cat{C'}$, and $\cat{D}$, an accessible functor $q\colon\cat{C}\set\to\cat{C'}\set$, and a pra-functor $p\colon\cat{C}\set\to\cat{D}\set$, the left Kan extension of $p$ along $q$
\[
  \lkan{\cat{C}\set}{p}{\cat{D}\set}{q}{\cat{C'}\set}
\]
is again a pra-functor $\lens{p}{q}\colon\cat{C'}\set\to\cat{D}\set$. Indeed, naturally in $X:\cat{C}\set$ and $d:\cat{D}$, we have by \cref{lemma.pra} that
\[
  p(X)(d)\cong\sum_{i:I_d}\cat{C}\set(A_{d,i},X),
\]
where $I\colon\cat{D}\to\smset$ and $A_{d,i}:\cat{C}\set$. The left Kan extension is given on $X':\cat{C'}\set$ by:
\[
\lens{p}{q}(X')(d)\cong\sum_{i:I_d}\cat{C'}\set(q(A_{d,i}),X')
\]
which is a pra-functor again by \cref{lemma.pra}.
\end{remark}

\begin{example}\label{ex.recover_site}
Recall from \cref{not.main} that if $c:\ob\cat{C}$ is an object in a small category, $\cat{C}[c]$ denotes the set of all outgoing maps.

Let $\elts[\cat{C}][\cat{C}]\colon\cat{C}\set\to\smset$ be given by $\elts[\cat{C}](X)\coloneqq\sum_{c:\ob\cat{C}}X(c)$, the objects in $X$'s category of elements. We can rewrite $X(c)=\cat{C}\set(\cat{C}(c,-),X)$, so by \cref{prop.pra_simple},
\begin{equation}\label{eqn.early_outfacing_poly}
\lens{\elts[\cat{C}]}{\elts[\cat{C}]}\cong\sum_{c:\ob\cat{C}}\yon^{\cat{C}[c]}\,,
\end{equation}
since $\elts[\cat{C}](\cat{C}(c,-))=\cat{C}[c]$. As we will see in \cref{ex.density_comonad}, the polynomial \eqref{eqn.early_outfacing_poly} carries a comonad structure that reconstructs $\cat{C}$ from the functor $\elts[\cat{C}]$.
\end{example}

\begin{example}\label{ex.monad_in_span}
A category $\cat{C}$ can be identified with a monad in the bicategory of set-spans, say $\ob\cat{C}\From{\dom} \Mor\cat{C}\To{\cod}\ob\cat{C}$. In turn, set-spans $A\from C\to B$ can be identified with left adjoint pra-functors of the form $A\set\to B\set$. Thus a category can be identified with a set $O\coloneqq\ob\cat{C}$ together with a left adjoint pra-functor monad $m\colon O\set\to O\set$; it sends $X$ to $o\mapsto\sum_{f\colon o'\to o}X(o')$. Using \cref{prop.right_adjoint_kan,rem.pra_general}, one could calculate $m$'s right adjoint $\lens{\id_{O\set}}{m}$, which will be a comonad on $O\set$.

Instead, let's consider the left Kan extension $\lens{\elts[O]}{\elts[O]\circ m}$, where $\elts[O]\colon O\set\to\smset$ is as in \cref{ex.recover_site}. By \cref{prop.pra_simple}, we calculate this to be (just as in \eqref{eqn.early_outfacing_poly}),
\[
\lens{\elts[O]}{\elts[O]\circ m}\cong\sum_{c:\ob\cat{C}}\yon^{\cat{C}[c]}.
\qedhere
\]
\end{example}

\begin{example}
The category of pra-functors $\cat{C}\set\to\cat{D}\set$ has a monoidal structure, defined in terms of left Kan extensions \cite{day1969enriched}. Its unit and the monoidal product of $p_1,p_2\colon\cat{C}\set\to\cat{D}\set$ are given as follows:
\[
	\lkan{1}{1}{\cat{D}\set}{1}{\cat{C}\set}
\qqand
	\lkan{(\cat{C}+\cat{C})\set}{p_1\times p_2}{\cat{D}\set}{(\times)}{\cat{C}\set}
\]
where for any $(X_1,X_2):\cat{C}\set\times\cat{C}\set\cong(\cat{C}+\cat{C})\set$, we have
\[
	(\times)(X_1,X_2)(c)\coloneqq X_1(c)\times X_2(c) 
  \qand
  (p_1\times p_2)(X_1,X_2)(d)\coloneqq p_1(X_1)(d)\times p_2(X_2)(d).
\qedhere
\]
\end{example}

\begin{example}\label{ex.simplicial}
Let $\cat{G}\coloneqq\begin{tikzcd}[ampersand replacement=\&, column sep=small, execute at end picture={\draw (\tikzcdmatrixname.north west) rectangle (\tikzcdmatrixname.south east);}]E\ar[r, shift left=3pt]\ar[r, shift right=3pt]\&V\end{tikzcd}$\, be the graph-indexing category, and let $\paths\colon\cat{G}\set\to\smset$ be the functor sending a graph to the set of all paths in it
\[
\paths(X)\coloneqq\sum_{n:\nn}\cat{G}\set(\ul{n},X),
\where[] \ul{n}\coloneqq\bullet\To{e_1}\cdots\To{e_n}\bullet.
\]
Let $\frcat\colon\cat{G}\set\to\cat{G}\set$ denote the usual free-category monad, so $\paths(G)=\frcat(G)(E)$. Then we compute
\[
\lens{\paths}{\paths\,\circ\,\frcat}\cong\sum_{n:\nn}\smset(\paths(\frcat(\ul{n})),-).
\]

Write $\simp$ for the category of finite ordered sets; its object set is in bijection with $\nn$. Then again with notation from \cref{not.main}, there is a natural isomorphism
\begin{equation}\label{eqn.simp}
  \lens{\paths}{\paths\,\circ\,\frcat}
  \cong
  \sum_{n:\ob\simp\op}\yon^{\simp\op[n]}.
\end{equation}
As we will see in \cref{ex.density_comonad}, the polynomial \eqref{eqn.simp} carries a comonad structure that reconstructs the simplicial indexing category $\simp\op$.
\end{example}

\chapter{Generalizing density comonads}\label{chap.density_comonad}

A right Kan extension of a functor $p\colon\cat{C}\to\cat{D}$ along itself is always a monad---called the codensity monad---and similarly a left Kan extension $\lens{p}{p}$ of $p$ along itself is a comonad---called the \emph{density comonad} \cite{leinster2012codensity,avery2016codensity,adamek2018formula,Abramsky2022discrete}; see \cref{ex.density_comonad}. In this section we generalize this construction in a number of ways, by involving monads in the contravariant position and comonads in the covariant position of $\lens{-}{-}$.

%two ways, corresponding intuitively to the two diagrams below:
%\begin{equation}\label{eqn.two_kans_intuitive}
%\begin{tikzcd}[column sep=large, row sep=25pt]
%	\cat{C}\ar[r, "r", ""' name=mapr]\ar[d, "p"']&\cat{C}\ar[r, "q", ""' name=mapq]\ar[d, "p" description]&\cat{C}\ar[d, "p"]\\
%	\cat{D}\ar[r, "\lens{p\circ r}{p}"', "" name=mapr']&\cat{D}\ar[r, "\lens{p\circ q}{p}"', "" name=mapq']&\cat{D}
%	\ar[from=mapr, to=mapr', Rightarrow, shorten=3pt]
%	\ar[from=mapq, to=mapq', Rightarrow, shorten=3pt]
%\end{tikzcd}
%\qqand
%\begin{tikzcd}[row sep=35pt]
%  \cat{C}\ar[r, "p", ""' name=mapp]\ar[d, "r"']&\cat{D}\ar[from=d, "\lens{p}{p\circ r}"']\\
%  \cat{C}\ar[r, "p" description, "" name=mapp1, ""' name=mapp2]\ar[d, "q"']&\cat{D}\ar[from=d, "\lens{p}{p\circ q}"']\\
%  \cat{C}\ar[r, "p"', "" name=mapp']&\cat{D}
%  \ar[from=mapp, to=mapp1, Rightarrow, shorten=6pt]
%  \ar[from=mapp2, to=mapp', Rightarrow, shorten=6pt]
%\end{tikzcd}
%\end{equation}
%The corresponding theorems are \cref{thm.colax,thm.colax_along}.

\begin{theorem}\label{thm.colax}
Given an accessible category $\cat{C}$, a cocomplete category $\cat{D}$, and accessible functor $p\colon\cat{C}\to\cat{D}$, the functor
\[
\lens{p\circ -}{p}\colon
	\Big(\End(\cat{C}), \id_\cat{C},\circ\Big)
	\to
	\Big(\End(\cat{D}), \id_\cat{D},\circ\Big)
\]
is colax monoidal. In particular, there are coherent counit and comultiplication maps:
\begin{equation}\label{eqn.colaxators}
  \lens{p}{p}\imp\id_\cat{C}.
  \qqand
  \lens{p\circ q\circ r}{p}\imp\lens{p\circ q}{p}\circ\lens{p\circ r}{p}.
\end{equation}
\end{theorem}
\begin{proof}
The former map of \eqref{eqn.colaxators} comes directly from \eqref{eqn.counit}. The latter is induced by applying \eqref{eqn.coclosure} to the map
\[
  p\circ q\circ r\imp
  \lens{p\circ q}{p}\circ p\circ r\imp
  \lens{p\circ q}{p}\circ\lens{p\circ r}{p}\circ p
\]
which is itself induced by two applications of \eqref{eqn.coeval}. We will prove coassociativity and leave counitality to the reader.

The coassociativity equation is that the natural transformation (written below using $pq$ rather than $p\circ q$, for typographical reasons) across the top-and-right is equal to that along the left-and-bottom of the following diagram, where every map is either the unit \eqref{eqn.coeval} or the counit \eqref{eqn.counit} of an adjunction like \eqref{eqn.global_Kan}; see \cref{fig.colax_coassociativity}, page \pageref{fig.colax_coassociativity}.
\begin{figure}
\[
\begin{tikzcd}[row sep=small, every arrow/.append style = {Rightarrow}]
	\lens{pqrs}{p}\ar[r]\ar[d]&
		\lens{\lens{pq}{p}prs}{p}\ar[r]\ar[d]&
			\lens{\lens{pq}{p}\lens{pr}{p}ps}{p}\ar[d]\ar[rd, equal]\\
	\lens{\lens{pqr}{p}ps}{p}\ar[r]\ar[d]&
		\lens{\lens{\lens{pq}{p}pr}{p}ps}{p}\ar[r]\ar[d]&
			\lens{\lens{\lens{pq}{p}\lens{pr}{p}p}{p}ps}{p}\ar[r]\ar[d]&
				\lens{\lens{pq}{p}\lens{pr}{p}ps}{p}\ar[d]\\
	\lens{\lens{pqr}{p}\lens{ps}{p}p}{p}\ar[r]\ar[d]&
		\lens{\lens{\lens{pq}{p}pr}{p}\lens{ps}{p}p}{p}\ar[r]\ar[d]&
			\lens{\lens{\lens{pq}{p}\lens{pr}{p}p}{p}\lens{ps}{p}p}{p}\ar[r]\ar[d]&
				\lens{\lens{pq}{p}\lens{pr}{p}\lens{ps}{p}p}{p}\ar[d]\\
	\lens{pqr}{p}\lens{ps}{p}\ar[r]&
		\lens{\lens{pq}{p}pr}{p}\lens{ps}{p}\ar[r]&
			\lens{\lens{pq}{p}\lens{pr}{p}p}{p}\lens{ps}{p}\ar[r]&
				\lens{pq}{p}\lens{pr}{p}\lens{ps}{p}
\end{tikzcd}
\]
\caption{Proof of coassociativity for the colax monoidal functor $\lens{p\circ -}{p}$ associated to any accessible $p\colon\cat{C}\to\cat{D}$ and cocomplete $\cat{D}$. See \cref{thm.monad_comonad_dist}.}\label{fig.colax_coassociativity}
\end{figure}
Every square in it is essentially an interchange isomorphism, and the triangle is \eqref{eqn.triangle}.
\end{proof}

Here is the setup for the following corollary, shown in two ways:
\[
\begin{tikzcd}
	\cat{C}\ar[r, "k"]\ar[d, "p"']&\cat{C}\ar[r, "p"]\ar[d, Rightarrow, shorten = 3pt]&\cat{D}\\
	\cat{D}\ar[urr, bend right, "\lens{p\circ k}{p}"', "" name=kan]&~
\end{tikzcd}
\hspace{.8in}
\begin{tikzcd}[column sep=large, row sep=25pt]
	\cat{C}\ar[r, "k", ""' name=mapr]\ar[d, "p"']&\cat{C}\ar[d, "p"]\\
	\cat{D}\ar[r, "\lens{p\circ k}{p}"', "" name=mapr']&\cat{D}
	\ar[from=mapr, to=mapr', Rightarrow, shorten=3pt]
\end{tikzcd}
\]
The corollary says that if $k$ is a comonad on $\cat{C}$, then $\lens{p\circ k}{p}$ is a comonad on $\cat{D}$.

\begin{corollary}\label{cor.cmd}
	Let $\cat{C}$ and $\cat{D}$ be accessible categories, with $\cat{D}$ cocomplete. For any accessible functor $p\colon\cat{C}\to\cat{D}$ there is a functor
	\[
	\lens{p\circ -}{p}\colon\comnd(\cat{C})\to\comnd(\cat{D})
	\]
	taking accessible comonads on $\cat{C}$ to accessible comonads on $\cat{D}$.
\end{corollary}
\begin{proof}
Follows from \cref{thm.colax}, since colax monoidal functors preserve comonoids.
\end{proof}

Recall from \cref{not.main} that $\then$ denotes the reverse of $\circ$.

\begin{theorem}\label{thm.colax_along}
Given an accessible category $\cat{C}$, a cocomplete category $\cat{D}$, and accessible functor $p\colon\cat{C}\to\cat{D}$, there is a colax monoidal functor
\[
\lens{p}{p\circ -}\colon
	\Big(\End(\cat{C})\op, \id_\cat{C},\circ\Big)
	\to
	\Big(\End(\cat{D}), \id_\cat{D},\then\Big)
\]
That is, there are coherent counit and comultiplication maps:
\begin{equation}\label{eqn.colax_along}
  \lens{p}{p}\imp\id_\cat{C}
  \qqand
  \lens{p}{p\circ q\circ r}\imp\lens{p}{p\circ q}\then\lens{p}{p\circ r}.
\end{equation}

\end{theorem}
\begin{proof}
Again, the former map is \eqref{eqn.counit} and the latter map is induced by
\[
  p\imp 
  \lens{p}{p\circ r}\circ p\circ r\imp 
  \lens{p}{p\circ r}\circ \lens{p}{p\circ q}\circ q\circ r.
\]
The proof is analogous to that of \cref{thm.colax}. 
\end{proof}

Here is the setup for the following corollary, shown in two ways
\[
\begin{tikzcd}
	\cat{C}\ar[r, "p", ""' name=p]\ar[d, "t"']&\cat{D}\\
	\cat{C}\ar[d, "p"']\\
	\cat{C}\ar[uur, bend right, "\lens{p}{p\circ t}"', "" name=kan]
	\ar[from=p, to=kan-|p, shorten <= 3pt, Rightarrow]
\end{tikzcd}
\hspace{.8in}
\begin{tikzcd}[row sep=35pt]
  \cat{C}\ar[r, "p", ""' name=mapp]\ar[d, "t"']&\cat{D}\ar[from=d, "\lens{p}{p\circ t}"']\\
  \cat{C}\ar[r, "p"', "" name=mapp']&\cat{D}
  \ar[from=mapp, to=mapp', Rightarrow, shorten=6pt]
\end{tikzcd}
\]
The corollary says that if $t$ is a monad on $\cat{C}$, then $\lens{p}{p\circ t}$ is a comonad on $\cat{D}$.

\begin{corollary}\label{cor.cmd_along}
	Let $\cat{C}$ and $\cat{D}$ be accessible categories, with $\cat{D}$ cocomplete. For any accessible functor $p\colon\cat{C}\to\cat{D}$ there is a functor
	\[
	\lens{p}{p\circ -}\colon\mnd(\cat{C})\op\to\comnd(\cat{D})
	\]
	taking accessible monads on $\cat{C}$ to accessible comonads on $\cat{D}$.
\end{corollary}
\begin{proof}
An accessible monad on $\cat{C}$ is a comonoid object in $\End(\cat{C})\op$, and a morphism of monads is a morphism in the opposite direction between the comonoids. By \cref{thm.colax_along}, monads on $\cat{C}$ are sent to comonoids in the reverse monoidal structure $(\End(\cat{D}),\id_\cat{D},\then)$, but these are the same as comonoids in $\End(\cat{D})$ with the usual $(\circ)$ monoidal structure.
\end{proof}

\begin{example}[Density comonad]\label{ex.density_comonad}
Applying either \cref{cor.cmd} or \cref{cor.cmd_along} to an accessible $p\colon\cat{C}\to\cat{D}$ and the identity endofunctor on $\cat{C}$---which is both a comonad and a monad---we obtain $\lens{p}{p}$, which is the associated density comonad.
\end{example}

\begin{remark}\label{rem.bicomodules}
Let $\cat{C},\cat{D}$ be as in \cref{cor.cmd,cor.cmd_along}. The functor $\lens{p\circ -}{p}$ (resp.\  $\lens{p}{p\circ -}$) sends bicomodules (resp.\ bimodules in the reverse direction) to bicomodules:
\begin{align*}%\label{eqn.bicomodule_bicomodule}
  (k_1\circ q\from q\to q\circ k_2)&\mapstoo
  \left(\lens{p\circ k_1}{p}\circ\lens{p\circ q}{p}\from
  		\lens{p\circ q}{p}\to
  			\lens{p\circ q}{p}\circ\lens{p\circ k_2}{p}\right)\\%\label{eqn.bimodule_bicomodule}
	(t_2\circ q\to q\from q\circ t_1)&\mapstoo
	\left(\lens{p}{p\circ t_1}\circ\lens{p}{p\circ q}\from
		\lens{p}{p\circ q}\to
			\lens{p}{p\circ q}\circ\lens{p}{p\circ t}\right)
\end{align*}
Also, for any $p,q\colon\cat{C}\to\cat{D}$, and accessible comonad $k:\comnd(\cat{C})$ (resp.\  accessible monad $t$), there is a bicomodule of the form
\begin{align}\label{eqn.bico_change_poly}
  \lens{p\circ k}{p}\circ\lens{p\circ k}{q}\fromm
  \lens{p\circ k}{q}&\too
  \lens{p\circ k}{q}\circ\lens{q\circ k}{q}\\
  \nonumber
  \left(\;\tn{resp.\quad }
  \lens{p}{p\circ t}\circ\lens{p}{q\circ t}\fromm
  \lens{p}{q\circ t}\right.&\left.\too
  \lens{p}{q\circ t}\circ\lens{q}{q\circ t}
  \;\right)
\end{align}
\end{remark}

Recall \cite{beck1969distributive,brookes1993monads} that a distributive law of a monad $t$ over a comonad $k$ is a map
$
	t\circ k\to k\circ t
$
satisfying ``the evident'' four laws, addressing (co)unit and (co)multiplication.

The following theorem generalizes \cref{cor.cmd,cor.cmd_along} since every monad distributes over the identity comonad and the identity monad distributes over every comonad. Here is a picture of the setup:
\[
\begin{tikzcd}
  \cat{C}\ar[r, "k", ""' name=topleft]\ar[d, "t"']&\cat{C}\ar[r, "k", ""' name=topright]\ar[d, "t"']&\cat{C}\ar[d, "p"]\\
  \cat{C}\ar[r, "k" description, "" name=midleft1, ""' name=midleft2]\ar[d, "t"']&\cat{C}\ar[d, "p"]&\cat{D}\\
  \cat{C}\ar[d, "p"']&\cat{D}\ar[ur, "\lens{p\circ k}{p\circ t}"', "" name=botright]\\
  \cat{D}\ar[ur, "\lens{p\circ k}{p\circ t}"', "" name=botleft]
  \ar[from=topleft, to=midleft1-|topleft, shorten=3pt, Rightarrow]
  \ar[from=midleft2-|topleft, to=botleft-|topleft, shorten=5pt, Rightarrow]
  \ar[from=topright, to=botright-|topright, shorten=5pt, Rightarrow]
\end{tikzcd}
\]

\begin{theorem}\label{thm.monad_comonad_dist}
If $t$ is an accessible monad on $\cat{C}$, $k$ is an accessible comonad on $\cat{C}$, $\alpha\colon t\circ k\to k\circ t$ is a distributive law, and $p\colon\cat{C}\to\cat{D}$ is any accessible endofunctor, then $\lens{p\circ k}{p\circ t}$ has the structure of a comonad on $\cat{D}$.
\end{theorem}
\begin{proof}
For brevity, we elide $\circ$'s in this proof. The counit $\lens{pk}{pt}\to\id$ is induced by the composite $pk\To{\epsilon} p\To{\eta} pt$. The comultiplication is induced by (co)monad maps, distributivity, and \eqref{eqn.coeval}:
\begin{align*}
	pk\To{\delta} 
	pkk&\to
	\lens{p k}{p t} p t k\To{\alpha}
	\lens{pk}{p t} p k t\to
	\lens{pk}{p t} \lens{p k}{p t} p t t\To{\mu}
	\lens{pk}{p t} \lens{p k}{p t} p t.
\end{align*}
\label{page.this}
We will prove coassociativity by a massive diagram chase---see \cref{fig.assoc}, page \pageref{fig.assoc}---and leave counitality to the reader.
\begin{figure}
\[
\scalebox{.6}{
\begin{tikzcd}[column sep=7pt, row sep=9pt, cramped, ampersand replacement=\&, every label/.append style={font=\fontsize{3}{3}\selectfont}]
%1
	\lc{pk}{pt}\ar[r]\ar[d]\&\lc{pkk}{pt}\ar[rrr]\ar[d]\&\&\&\lc{\lc{pk}{pt}ptk}{pt}\ar[r]\ar[d]\&\lc{\lc{pk}{pt}pkt}{pt}\ar[r]\ar[d]\&\lc{\lc{pk}{pt}\lc{pk}{pt}ptt}{pt}\ar[r]\ar[d]\&\lc{\lc{pk}{pt}\lc{pk}{pt}pt}{pt}\ar[r]\ar[d]\&\lc{pk}{pt}\lc{pk}{pt}\ar[d]
	\\
%2
	\lc{pkk}{pt}\ar[r]\ar[d]\&\lc{pkkk}{pt}\ar[rrr]\ar[d]\&\&\&\lc{\lc{pkk}{pt}ptk}{pt}\ar[r]\ar[d]\&\lc{\lc{pkk}{pt}pkt}{pt}\ar[r]\ar[d]\&\lc{\lc{pkk}{pt}\lc{pk}{pt}ptt}{pt}\ar[r]\ar[d]\&\lc{\lc{pkk}{pt}\lc{pk}{pt}pt}{pt}\ar[r]\ar[d]\&\lc{pkk}{pt}\lc{pk}{pt}\ar[d]
	\\
%3
	\lc{\lc{pk}{pt}ptk}{pt}\ar[r]\ar[dd]\&\lc{\lc{pk}{pt}ptkk}{pt}\ar[rrr]\ar[d]\&\&\&\lc{\lc{\lc{pk}{pt}ptk}{pt}ptk}{pt}\ar[r]\ar[d]\&\lc{\lc{\lc{pk}{pt}ptk}{pt}pkt}{pt}\ar[r]\ar[d]\&\lc{\lc{\lc{pk}{pt}ptk}{pt}\lc{pk}{pt}ptt}{pt}\ar[r]\ar[d]\&\lc{\lc{\lc{pk}{pt}ptk}{pt}\lc{pk}{pt}pt}{pt}\ar[r]\ar[d]\&\lc{\lc{pk}{pt}ptk}{pt}\lc{pk}{pt}\ar[d]
	\\
%4
	\ar[r, phantom, "\text{\tiny Dist}"]\&\lc{\lc{pk}{pt}pktc}{pt}\ar[d]\ar[rd]\ar[rrr]\&~\&\&\lc{\lc{\lc{pk}{pt}pkt}{pt}ptk}{pt}\ar[r]\ar[d]\&\lc{\lc{\lc{pk}{pt}pkt}{pt}pkt}{pt}\ar[r]\ar[d]\&\lc{\lc{\lc{pk}{pt}pkt}{pt}\lc{pk}{pt}ptt}{pt}\ar[r]\ar[d]\&\lc{\lc{\lc{pk}{pt}pkt}{pt}\lc{pk}{pt}pt}{pt}\ar[r]\ar[d]\&\lc{\lc{pk}{pt}pkt}{pt}\lc{pk}{pt}\ar[d]
	\\
%5
	\lc{\lc{pk}{pt}pkt}{pt}\ar[r]\ar[d]\&\lc{\lc{pk}{pt}pkkt}{pt}\ar[rd]\ar[d]\&\lc{\lc{pk}{pt}\lc{pk}{pt}pttk}{pt}\ar[rr]\ar[d]\ar[ddrr]\&\&\lc{\lc{\lc{pk}{pt}\lc{pk}{pt}ptt}{pt}ptk}{pt}\ar[r]\ar[d]\&\lc{\lc{\lc{pk}{pt}\lc{pk}{pt}ptt}{pt}pkt}{pt}\ar[r]\ar[d]\&\lc{\lc{\lc{pk}{pt}\lc{pk}{pt}ptt}{pt}\lc{pk}{pt}ptt}{pt}\ar[r]\ar[d]\&\lc{\lc{\lc{pk}{pt}\lc{pk}{pt}ptt}{pt}\lc{pk}{pt}pt}{pt}\ar[r]\ar[d]\&\lc{\lc{pk}{pt}\lc{pk}{pt}ptt}{pt}\lc{pk}{pt}\ar[d]
	\\
%6
	\lc{\lc{pk}{pt}\lc{pk}{pt}ptt}{pt}\ar[r]\ar[dddd]\&\lc{\lc{pk}{pt}\lc{pkk}{pt}ptt}{pt}\ar[rddd]\ar[dddd]\&\lc{\lc{pk}{pt}\lc{pk}{pt}ptkt}{pt}\ar[rrdd]\ar[ddd]\ar[rrd, phantom, "\text{\tiny Dist}"]\&\&\lc{\lc{\lc{pk}{pt}\lc{pk}{pt}pt}{pt}ptk}{pt}\ar[r]\ar[d]\&\lc{\lc{\lc{pk}{pt}\lc{pk}{pt}pt}{pt}pkt}{pt}\ar[r]\ar[d]\&\lc{\lc{\lc{pk}{pt}\lc{pk}{pt}pt}{pt}\lc{pk}{pt}ptt}{pt}\ar[r]\ar[d]\&\lc{\lc{\lc{pk}{pt}\lc{pk}{pt}pt}{pt}\lc{pk}{pt}pt}{pt}\ar[r]\ar[dddd]\&\lc{\lc{pk}{pt}\lc{pk}{pt}pt}{pt}\lc{pk}{pt}\ar[dddd]
	\\[10pt]
%7
	\&\&~\&\&\lc{\lc{pk}{pt}\lc{pk}{pt}ptk}{pt}\ar[r]\&\lc{\lc{pk}{pt}\lc{pk}{pt}pkt}{pt}\ar[r]\&\lc{\lc{pk}{pt}\lc{pk}{pt}\lc{pk}{pt}ptt}{pt}\ar[dddr]\ar[dd]
	\\
%8
 \&\&\&\&\lc{\lc{pk}{pt}\lc{pk}{pt}pktt}{pt}\ar[r]\ar[ru]\ar[ld]\&\lc{\lc{pk}{pt}\lc{pk}{pt}\lc{pk}{pt}pttt}{pt}\ar[ru]\ar[ld]\&
	\\
%9
	\&\&\lc{\lc{pk}{pt}\lc{\lc{pk}{pt}ptk}{pt}ptt}{pt}\ar[r]\ar[d]\&\lc{\lc{pk}{pt}\lc{\lc{pk}{pt}pkt}{pt}ptt}{pt}\ar[r]\ar[d]\&\lc{\lc{pk}{pt}\lc{\lc{pk}{pt}\lc{pk}{pt}ptt}{pt}ptt}{pt}\ar[rr]\ar[d]\&\&\lc{\lc{pk}{pt}\lc{\lc{pk}{pt}\lc{pk}{pt}pt}{pt}ptt}{pt}\ar[d]
	\\
%10
	\lc{\lc{pk}{pt}\lc{pk}{pt}pt}{pt}\ar[r]\ar[d]\&\lc{\lc{pk}{pt}\lc{pkk}{pt}pt}{pt}\ar[r]\ar[d]\&\lc{\lc{pk}{pt}\lc{\lc{pk}{pt}ptk}{pt}pt}{pt}\ar[r]\ar[d]\&\lc{\lc{pk}{pt}\lc{\lc{pk}{pt}pkt}{pt}pt}{pt}\ar[r]\ar[d]\&\lc{\lc{pk}{pt}\lc{\lc{pk}{pt}\lc{pk}{pt}ptt}{pt}pt}{pt}\ar[rr]\ar[d]\&\&\lc{\lc{pk}{pt}\lc{\lc{pk}{pt}\lc{pk}{pt}pt}{pt}pt}{pt}\ar[r]\ar[d]\&\lc{\lc{pk}{pt}\lc{pk}{pt}\lc{pk}{pt}pt}{pt}\ar[r]\ar[d]\&\lc{pk}{pt}\lc{pk}{pt}\lc{pk}{pt}\ar[dl, equal]
	\\
%11
	\lc{pk}{pt}\lc{pk}{pt}\ar[r]\&\lc{pk}{pt}\lc{pkk}{pt}\ar[r]\&\lc{pk}{pt}\lc{\lc{pk}{pt}ptk}{pt}\ar[r]\&\lc{pk}{pt}\lc{\lc{pk}{pt}pkt}{pt}\ar[r]\&\lc{pk}{pt}\lc{\lc{pk}{pt}\lc{pk}{pt}ptt}{pt}\ar[rr]\&\&\lc{pk}{pt}\lc{\lc{pk}{pt}\lc{pk}{pt}pt}{pt}\ar[r]\&\lc{pk}{pt}\lc{pk}{pt}\lc{pk}{pt}
\end{tikzcd}
}
\]
\caption{Proof of coassociativity for the comonad $\lens{p\circ k}{p\circ t}$ associated to any accessible $p\colon\cat{C}\to\cat{D}$, cocomplete $\cat{D}$, and (co)monad $k,t\colon\cat{C}\to\cat{C}$. See \cref{thm.monad_comonad_dist}.}
\label{fig.assoc}
\end{figure}
\end{proof}

\chapter{Categories by Kan extension}

As mentioned in the introduction (\cref{chap.intro}), polynomial comonads can be identified with small categories.
\begin{proposition}[ \cite{ahman2016directed}]\label{prop.au}
Categories, up to isomorphism, can be identified with polynomial comonads on $\smset$.
\end{proposition}

We now briefly explain how this works. For any category $\cat{C}$, the associated polynomial comonad is carried by the functor
\begin{equation}\label{eqn.cat_comonad_carrier}
\catcmd{C}\coloneqq\sum_{c:\ob\cat{C}}\yon^{\cat{C}[c]}\colon\smset\to\smset,
\end{equation}
where $\cat{C}[c]$ is the set of outgoing maps from $c$ in $\cat{C}$; see \cref{not.main}. We will see that the counit $\epsilon\colon\catcmd{C}\to\yon$ encodes the identity morphisms and the comultiplication $\delta\colon\catcmd{C}\to\catcmd{C}\circ\catcmd{C}$ encodes the codomains and the compositions.

We have seen this polynomial before: it is $\catcmd{C}\cong\lens{\elts[\cat{C}]}{\elts[\cat{C}]}$ from \cref{ex.recover_site}.%
\footnote{We also obtained this polynomial from the monad-on-set-spans version of $\cat{C}$ in \cref{ex.monad_in_span}.
}
We thus formally obtain the above-mentioned maps $\epsilon\colon\catcmd{C}\to\id_\smset$ and $\delta\colon\catcmd{C}\to\catcmd{C}\circ\catcmd{C}$ from \cref{ex.density_comonad}. Explicitly, 
$
\epsilon_X\colon\sum_{c:\ob\cat{C}}\smset(\cat{C}[c],X)\to X
$
is given by $(c,x)\mapsto x(\id_c)$, and
\begin{equation}\label{eqn.delta}
	\delta_C\colon\sum_{c:\ob\cat{C}}\smset(\cat{C}[c],X)\to\sum_{c:\ob\cat{C}}\smset\left(\cat{C}[c],\sum_{c:\ob\cat{C}}\smset(\cat{C}[c],X)\right)
\end{equation}
is given by $(c,x)\mapsto (c, f\mapsto(\cod(f), f'\mapsto x(f\then f')))$.%
\footnote{In case the typing isn't clear, $c:\ob\cat{C}$, $x:\cat{C}[c]\to X$, $f\colon\cat{C}[c]$, and $f'\colon\cat{C}[\cod(f)]$.
}

So every small category $\cat{C}$ can be obtained in this way by the left Kan extension of a pra-functor (from a copresheaf category to $\smset$) along itself. 

\begin{remark}\label{rem.retro}
Note that comonad morphisms between polynomial comonads are not functors, but \emph{retrofunctors}. All of the constructions (e.g.\ \cref{prop.selection,prop.prod_completion,prop.lawvere}) in this section are functorial as maps into the category of categories and retrofunctors, but retrofunctors are beyond the scope of this paper.
\end{remark}

\begin{remark}\label{rem.pra_bico}
Garner showed that polynomial bicomodules $\catcmd{C}\circ p\from p\to p\circ \catcmd{C'}$ can be identified with pra-functors $\cat{C'}\set\to\cat{C}\set$; see \href{https://www.youtube.com/watch?v=tW6HYnqn6eI}{his HoTTEST video} or \cite{spivak2025functorial}. While many of our results extend to this setting (see e.g.\ \cref{rem.bicomodules}), in order to keep this work as self-contained as possible, we do not discuss it further here.
\end{remark}

\begin{proposition}\label{prop.selection}
Let $p\colon\smset\to\smset$ be a polynomial functor. Then we have a function
\[
\lens{p\circ-}{p}\colon\ob\smcat\to\ob\smcat.
\]
\end{proposition}
\begin{proof}
This follows from \cref{cor.cmd,prop.au,prop.pra_simple}.
\end{proof}

We will see that the function in \cref{prop.selection} extends to a functor $\lens{p\circ-}{p}\colon\smcat\to\smcat$, in fact one that respects the bijective-on-objects, fully-faithful factorization system; see \cref{cor.sel_functorial,prop.selection_boff}.%
\footnote{
This map also extends to a map on bicomodules; see \cref{rem.pra_bico,rem.bicomodules}.
}
We will refer to categories of the form $\lens{p\circ\catcmd{C}}{p}$ as \emph{selection categories}.

Given a small universe $\Cat{U}$ of sets, where the set associated to any $A:\Cat{U}$ is denoted $[A]$, define the polynomial
\[\uu\coloneqq\sum_{A:\Cat{U}}\yon^{[A]}.\]
We will take this $\uu$ as notation going forward. Sometimes we write $A$ rather than $[A]$.

\begin{proposition}\label{prop.prod_completion}
The category $\lens{\uu\circ\catcmd{C}}{\uu}$ is the $\Cat{U}$-ary product completion of $\cat{C}$.
\end{proposition}
\begin{proof}
Let $\cat{C}$ be a small category. An object of $\lens{\uu\circ\catcmd{C}}{\uu}$ is a set $A:\Cat{U}$ and $A$-many elements of $\ob\cat{C}$; denote it $(A,c)$, where $c\colon A\to\ob\cat{C}$. A morphism out of $(A,c)$ consists of a set $A':\Cat{U}$ and $A'$-many elements of $\sum_{a:A}\cat{C}[c_a]$. These are the objects and outgoing morphisms of the $\Cat{U}$-ary product completion, and we leave the rest to the reader.
\end{proof}

\begin{example}
The polynomial associated to the universe of finite sets is $\List\coloneqq\sum_{N:\nn}\yon^N$, and for any category $\cat{C}$, the comonad $\lens{\List\circ\catcmd{C}}{\List}$ can be identified with the finite product completion of $\cat{C}$. Adding a free terminal object to $\cat{C}$ is given by $\lens{(\yon+1)\circ\catcmd{C}}{\yon+1}=\lens{\catcmd{C}+1}{\yon+1}$.
\end{example}

Selection categories---i.e.\ those of the form $\lens{p\circ\catcmd{C}}{p}$ as in \cref{prop.selection}---appear not to have been studied in generality before, and the construction is a main subject of the remainder of this paper; we will return to them in \cref{ex.internal_cats} after discussing some examples of other results from \cref{chap.density_comonad}.

\begin{proposition}\label{prop.lawvere}
Let $p\colon\cat{C}\to\smset$ be a pra-functor. Then we have a function%
\footnote{For additional functoriality of $\lens{p}{p\circ -}$, see \cref{rem.retro}. See also \cref{rem.bicomodules}.}
\[
  \lens{p}{p\circ -}\colon\ob\mnd(\cat{C})\to\ob\smcat.
\]
\end{proposition}
\begin{proof}
This follows from \cref{cor.cmd_along,prop.au,prop.pra_simple}.
\end{proof}

\begin{example}\label{ex.lawvere}
To every $S$-sorted $\Cat{U}$-ary Lawvere theory $\cat{T}$ is an associated monad $t$ on $S\set$; consider the associated Kleisli category $S\set_t$. Then $\cat{T}$ can be identified with the opposite of the full subcategory of $S\set_t$ spanned by $\Cat{U}$-small $S$-sets. This is a bit of a mouthful, but it arises as a left Kan extension:
\begin{equation}\label{eqn.lawvere}
\cat{T}\cong\lens{\uu\,\circ\,\elts[S]}{\uu\,\circ\,\elts[S]\,\circ\, t}.
\end{equation}
%where $\uu_S\colon S\set\to\smset$ is given by
%\[\uu_S\coloneqq u\circ\elts[S]\cong\sum_{A:\Cat{U}}\sum_{s\colon[A]\to S}S\set\left(\sum_{a:[A]}\{s(a)\},-\right).\]
One can check by reading off the objects and morphisms of the corresponding category using \eqref{eqn.kan_pra_to_set,eqn.cat_comonad_carrier}. For example, if $t$ is finitary and $S=1$, then $\cat{T}\cong\lens{\List}{\List\circ t}$. 
\end{example}

\begin{example}\label{ex.walking}
For any monad $t\colon\smset\to\smset$, set $I$, and sets $A\colon I\to\smset$, the full subcategory of the Kleisli category $\smset_t\op$ spanned by $(A_i)_{i:I}$ can be identified with the comonad $\lens{p}{p\circ t}$, where $p\coloneqq\sum_{i:I}\yon^{A_i}$. For example, the walking arrow is
\[
  \fbox{$\bullet\to\bullet$}\cong\lens{\yon^1+\yon^0}{\yon^1+\yon^0}
  \qedhere
\]
\end{example}

A variant of the following appeared in \cite[Example 3.9]{shapiro2024polynomial}, where  theory categories $\Theta_m$ for monads $m$ associated to various higher categories, together with the associated nerve functors, were discussed.

\begin{proposition}\label{prop.deltaop}
The polynomial comonad associated by \cref{cor.cmd_along} to the left Kan extension
\[
\begin{tikzcd}[row sep=35pt]
  \cat{G}\set\ar[r, "\paths", ""' name=mapp]\ar[d, "\frcat"']&\smset\ar[from=d, "\lens{\paths}{\paths\,\circ\, \frcat}"']\\
  \cat{G}\set\ar[r, "\paths"', "" name=mapp']&\smset
  \ar[from=mapp, to=mapp', Rightarrow, shorten=6pt]
\end{tikzcd}
\]
is isomorphic the simplicial indexing category
\[\lens{\paths}{\paths\,\circ\, \frcat}\cong\simp\op.\]
\end{proposition}
\begin{proof}
We saw in \eqref{eqn.simp} that the carriers agree, and we may apply \cref{cor.cmd_along} because the free category construction $\frcat\colon\cat{G}\set\to\cat{G}\set$ carries an accessible (see \cref{lemma.pra_accessible}) monad structure. So we now have two comonad structures on $\catcmd{\simp\op}=\sum_{n:\nn}\yon^{\simp\op[n]}$ and we leave it to the reader to show that they agree. 
\end{proof}

\begin{example}
Every map $p\circ q\to q\circ p$ of polynomials induces a distributive law $\free_p\circ\cofree_q\to\cofree_q\circ\free_p$ of the free monad $\free_p$ over the cofree comonad $\cofree_q$ \cite{klin2011bialgebras}.

Note that every coalgebra $\alpha\colon S\to p(S)$ is an example: since $S=S\yon^0$ is a polynomial, $\alpha$ can be identified with a map $S\circ p\to p\circ S$. The cofree comonad $\cofree_p$ corresponds to the category with $p$-behavior trees (\cite{niu2024polynomial,jacobs2017introduction}) as objects and finite paths up the tree as morphisms; using $\alpha$, each $s:S$ corresponds to such a tree; in fact $\cofree_p(1)$ is the terminal $p$-coalgebra. The free monad $\free_S\cong\yon+S$ is the $S$-exceptions monad. 

The category $\lens{\cofree_p}{\free_S}$ defined in \cref{thm.monad_comonad_dist} again has $p$-behavior trees as objects, but a map is either a finite path up the tree or an exception $s:S$, which immediately lands in the behavior tree associated to $s$.
\end{example}

We now return to selection categories, i.e.\ those of the form $\lens{p\circ\catcmd{C}}{p}$ from \cref{prop.selection}, which will be our focus for the remainder of the paper.

\begin{example}\label{ex.internal_cats}
It was shown in \cite[Secion 5.1]{shapiro2023structures} that internal categories in $\cat{C}\set$ can be identified with pra-functor comonads on $\cat{C}\set$. Thus for any pra-functor $p\colon\cat{C}\set\to\smset$ and any internal category $k$ in $\cat{C}\set$, there is a category $\lens{p\circ k}{p}$.
\end{example}

\begin{lemma}\label{lemma.selection}
Let $\cat{C}$ be a category with associated comonad $\catcmd{C}$, and let $p\coloneqq\sum_{i:I}\yon^{A_i}$ be a polynomial functor. The following categories are isomorphic:
\begin{enumerate}[label=(\alph*)]
	\item the category associated to the comonad $\lens{p\circ\catcmd{C}}{p}$.
	\item the category $\cat{S}_\cat{C}(p)$ of $p$-ary formal products in $\cat{C}$. In other words, its object set is given by 
	\begin{equation}\label{eqn.ob_select}
		\ob\cat{S}_\cat{C}(p)\coloneqq\sum_{i:I}\prod_{a:A_i}\ob\cat{C},
	\end{equation}
	i.e.\ an object is a pair $(i,c)$, where $i:I$ and $c\colon A_i\to\ob\cat{C}$, and its hom-sets are given by
	\begin{equation}\label{eqn.hom_sets_selection}
	\Hom_{\cat{S}_\cat{C}(p)}\big((i,c),(i',c')\big)\coloneqq\prod_{a':A_{i'}}\sum_{a:A_i}\cat{C}\left(c_a,c'_{a'}\right)
	\end{equation}
	with the obvious identities and compositions.
\end{enumerate}
\end{lemma}
\begin{proof}
From \eqref{eqn.kan_pra_to_set} and \eqref{eqn.cat_comonad_carrier} we know that the object-set of $\lens{p\circ\catcmd{C}}{p}$ is given by $\sum_{i:I}\prod_{a:A_i}\ob\cat{C}$, as required, and that for any choice $(i,c)$ of object, the set of outgoing maps is given by
\[
	\lens{p\circ\catcmd{C}}{p}[(i,c)]\cong\sum_{i':I}\prod_{a':A_{i'}}\sum_{a:A_i}\cat{C}[c_i].
\]
From the definition \eqref{eqn.delta} of the comultiplication $\delta$, we also know that the codomain of $(i',a,f)$, is $(i', \cod(f))$,%
\footnote{
The typing is: $i':I$, $a:\prod_{a':A_{i'}}A_i$, and $f:\prod_{a':A_{i'}}\cat{C}[c_i]$.} 
so the fiber over $(i',c')$ is as in \eqref{eqn.hom_sets_selection}. We leave the identity and composite information to the reader.
\end{proof}

\begin{definition}[Selection category]
We say that a category $\cat{S}$ is a \emph{selection category} if there exists a polynomial $p$ and a category $\cat{C}$ such that $\cat{S}$ is isomorphic to either (hence both) of the categories described in \cref{lemma.selection}.
\end{definition}

Selection categories were discussed in a blog post titled \href{https://topos.institute/blog/2021-12-30-selection-categories/}{Creating new categories from old: Selection categories}, where a variety of results were announced, but not proven; we will prove some of them in \cref{chap.sel_theory}. We end this section with two examples.

\begin{example}\label{ex.selection}
Suppose $p\coloneqq\{S_1, S_2\}\yon^5+\{S_3\}\yon^4+\{S_4,S_5\}\yon^3$ is some arbitrary polynomial. Let $\cat{C}$ be some category. Then here is a picture of six composable morphisms in the selection category $\cat{S}_\cat{C}(p)=\lens{p\circ\catcmd{C}}{p}$ (left) and their composite (right):
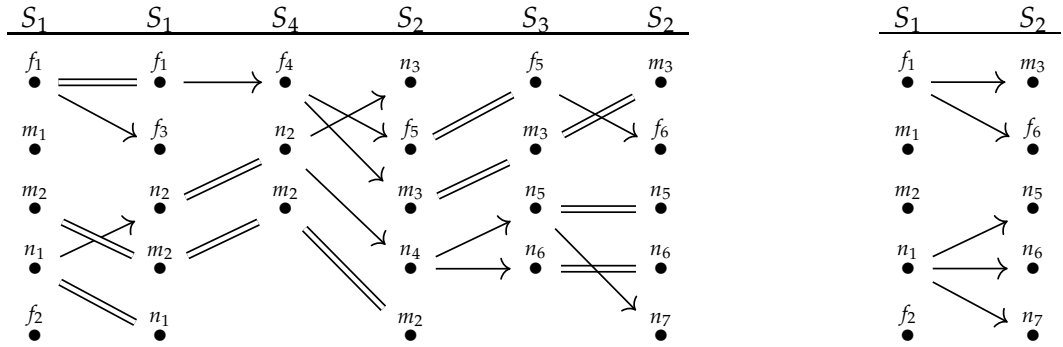
\[
\begin{tikzcd}[row sep=1pt]
	S_1&S_1&S_4&S_2&S_3&S_2\\[-5pt]\hline\\
	\bul{f_{1}}&\bul{f_{1}}\ar[from=l, equal] &\bul{f_4}\ar[from=l]&
		\bul{n_3}\ar[from=dl]&\bul{f_5}\ar[from=dl, equal]&\bul{m_3}\ar[from=dl, equal]\\
	\bul{m_{1}}&\bul{f_{3}}\ar[from=ul]&\bul{n_2}\ar[from=dl, equal] &
		\bul{f_5}\ar[from=ul]&\bul{m_3}\ar[from=dl, equal]&\bul{f_6}\ar[from=ul]\\
	\bul{m_{2}}&\bul{n_{2}}\ar[from=dl]&\bul{m_2}\ar[from=dl, equal]&
		\bul{m_3}\ar[from=uul]&\bul{n_5}\ar[from=dl]&\bul{n_5}\ar[from=l, equal]\\
	\bul{n_{1}}&\bul{m_{2}}\ar[from=ul, equal]&                   &
		\bul{n_4}\ar[from=uul]&\bul{n_6}\ar[from=l]&\bul{n_6}\ar[from=l, equal]\\
	\bul{f_{2}}&\bul{n_{1}}\ar[from=ul, equal]&                   &
		\bul{m_2}\ar[from=uul, equal]&                &\bul{n_7}\ar[from=uul]
\end{tikzcd}
\hspace{1in}
\begin{tikzcd}[row sep=1pt]
	S_1&S_2\\[-5pt]\hline\\
	\bul{f_{1}}&\bul{m_3}\ar[from=l]\\
	\bul{m_{1}}&\bul{f_6}\ar[from=ul]\\
	\bul{m_{2}}&\bul{n_5}\ar[from=dl]\\
	\bul{n_{1}}&\bul{n_6}\ar[from=l]\\
	\bul{f_{2}}&\bul{n_7}\ar[from=ul]
\end{tikzcd}
\]

Thinking in terms of natural selection,%
\footnote{One could imagine each object $c:\ob\cat{C}$ as an organism with one of three sexes: male, female, or nonsexed. In $\cat{C}$, perhaps each morphism out of a male is the identity and each morphism out of a female has an associated male and its codomain is either male or female.}
at first there were 5 slots (niches), occupied by two females, two males, and a nonsexed organism. Between the first and second time-steps, one male and one female died, the other male and female survived and had a female offspring, and the nonsexed survived and reproduced. The first male survived until time-step 4, and so on.%
%\footnote{To decide on these maps---which organisms reproduce in the context of other organisms and the world in some state---one could use a retrofunctor from $\Cat{Sel}_p$ to some ``nature'' category, but retrofunctors are beyond the scope of this paper.}

Selection categories can also be used to model states and memory locations in a computational process. A similar perspective---in the context of computational effects handling---was given in \cite[Section 5.2]{lynch2023concepts}.
\end{example}

\begin{example}
Using nonidentity monads, we could get even more interesting behavior using \cref{thm.monad_comonad_dist}. For example, the monad of finite random variables is
\[\mathit{rand}\coloneqq\sum_{N:\nn}\sum_{P:\Delta_N}\yon^N,\]
where $\Delta_N\coloneqq\{P:\{\mathsf{1},\ldots,\mathsf{N}\}\to[0,1]\mid P(\mathsf{1})+\cdots+P(\mathsf{N})=1\}$ is the distributions on an $N$-element set. 

Say that a category $\cat{C}$ is \emph{convex} if there is a distributive law $\mathit{rand}\circ\catcmd{C}\to\catcmd{C}\circ\mathit{rand}$; in the case $\cat{C}$ is a discrete, this is the same as the set $\ob\cat{C}$ having a convex structure in the usual sense. Given a convex category, the category $\lens{p\circ\catcmd{C}}{p\,\circ\,\mathit{rand}}$ is like that from \cref{ex.selection}, but rather than each slot picking a slot at the previous time step and a $\cat{C}$-map out of the object there, it picks a \emph{random} slot from the previous time step and a \emph{random} $\cat{C}$-map out of the object there.
\end{example}

\chapter{Basic theory of selection categories}\label{chap.sel_theory}

While selection categories appear not to have been studied before, from the perspective of this paper they are as natural as---and in some sense dual to---Lawvere theories: if $k$ is a polynomial comonad and $t$ is an arbitrary monad, one is $\lens{p\circ k}{p}$ and the other is $\lens{p}{p\circ t}$. We cannot justify their importance much beyond this---perhaps overly-facile---observation. The only other evidence that they are interesting theoretically is the fact that the construction is well-behaved as we show in \cref{thm.sel_enriched,prop.selection_boff}. We do not include the more general story from \cref{thm.monad_comonad_dist} here, leaving the right way to phrase that as an open question.

Let $\poly$ denote the category of polynomial functors $\smset\to\smset$ and natural transformations between them.

Let $\pprofcat$ denote the category of small categories enriched in the double category $\pprof$. Recall that a category enriched in $\pprof$ consists of a set $P$ of objects, for each $p:P$ a category $S(p)$, and for each pair of objects $p,q$, a profunctor $S(p)\tickar S(q)$, as well as identity and composite data. Each such is an object of $\pprofcat$, and a morphism between two such is a $\pprof$-enriched functor.

\begin{theorem}\label{thm.sel_enriched}
The selection category construction (\cref{lemma.selection}) comprises a functor
\[
	\cat{S}_-\colon\smcat\too\pprofcat.
\]
For each category $\cat{C}:\ob\smcat$, the associated $\pprof$-enriched category $\cat{S}_\cat{C}$ has objects $\ob\cat{S}_\cat{C}\coloneqq\ob\poly$. Assigned to $p:\poly$ is the selection category 
$
\cat{S}_\cat{C}(p)\cong\lens{p\circ\catcmd{C}}{p}
$
and assigned to each pair of objects $(p,q)$%
\footnote{Notation: $p\coloneqq\sum_{i:I}\yon^{A_i}$ and $q\coloneqq\sum_{j:J}\yon^{B_{j}}$, so $\ob\cat{S}_p\cong\sum_{i:I}\prod_{a:A_i}\ob\cat{C}$, and similarly for $\ob\cat{S}_{q}$.}
is the profunctor
\begin{align}\nonumber
  \Hom_{\cat{S}_\cat{C}}(p,q)\colon\cat{S}_\cat{C}(p)\op\times\cat{S}_\cat{C}(q)&\to\smset\\\label{eqn.prof_desc}
  \big((i,c)\;\;,\;\; (j,d)\big)&\mapsto\prod_{b:B_{j}}\sum_{a:A_i}\cat{C}\big(c(a),d(b)\big).
\end{align}
\end{theorem}
\begin{proof}
We first check that the on-objects description of the profunctor, given in \eqref{eqn.prof_desc},%
\footnote{
By \cref{rem.pro_is_pra}, profunctors are special pra-functors, and by \cref{rem.pra_bico}, pra-functors between copresheaf categories are polynomial bicomodules between the associated polynomial comonads. Under this correspondence, the profunctors $\Hom_{\cat{S}_\cat{C}}(p,q)$ defined here \eqref{eqn.prof_desc} correspond to the bicomodules from \eqref{eqn.bico_change_poly}.
}
 is functorial in the morphisms from \eqref{eqn.hom_sets_selection}. The map
\begin{multline*}
  \bigg(\prod_{a:A_i}\sum_{a':A'_{i'}}\cat{C}\big(c'(a'),c(a)\bigg)\times
  \bigg(\prod_{b:B_{j}}\sum_{a:A_i}\cat{C}\big(c(a),d(b)\big)\bigg)\times
  \bigg(\prod_{b':B_{j'}}\sum_{b:B_j}\cat{C}\big(d(b),d'(b')\big)\bigg)\\\too
  \bigg(\prod_{b':B'_{j'}}\sum_{a':A'_{i'}}\cat{C}\big(c(a'),d(b')\big)\bigg)
\end{multline*}
is given by $((a',f),(a,g),(b,h))\mapsto b'\mapsto ((a'\circ a\circ b)(b'), f\then g\then h).$%
\footnote{
The typing is given by $a'\colon A_i\to A'_{i'}$ and $f:\prod_{a:A_i}\cat{C}(c'(a'(a)),c(a))$, etc.
}
The unit map $\id_{\cat{S}_\cat{C}(p)}\imp\Hom_{\cat{S}_\cat{C}}(p,p)$  is itself the identity. We leave the composition map, as well as unitality and associativity, to the reader.

Suppose given a functor $F\colon\cat{C}\to\cat{C'}$. We will provide an identity-on-objects $\pprof$-enriched functor $\cat{S}_F\colon\cat{S}_\cat{C}\to\cat{S}_\cat{C'}$. For each object $p:\poly$, the functor $\cat{S}_F(p)\colon\cat{S}_\cat{C}(p)\to\cat{S}_\cat{C'}(p)$ sends object $(i,c)$ to $(i,F(c))$, where $F(c)(a)\coloneqq F(c(a))$ for any $a:A_i$, and similarly sends morphism $(a,f)\colon (i,c)\to(i',c')$ to $(a,F(f))$; see \cref{eqn.ob_select,eqn.hom_sets_selection} for notation. 

Finally, for each $p,q$, we need a profunctor morphism of the form
\[
\begin{tikzcd}[column sep=60pt]
	\cat{S}_\cat{C}(p)\ar[r, tick, "{\Hom_{\cat{S}_\cat{C}}(p,q)}", ""' name=top]\ar[d, "\cat{S}_F(p)"']&
		\cat{S}_\cat{C}(q)\ar[d, "\cat{S}_F(q)"]\\
		\cat{S}_\cat{C'}(p)\ar[r, tick, "{\Hom_{\cat{S}_\cat{C'}}(p,q)}"', "" name=bot]&\cat{S}_\cat{C'}(q)
		\ar[from=top, to=bot, Rightarrow, shorten=3pt, "{\cat{S}_F(p,q)}"]
\end{tikzcd}
\]
It is again entirely straightforward. Indeed, given objects $(i,c):\ob\cat{S}_\cat{C}(p)$ and $(j,d):\ob\cat{S}_\cat{C}(q)$, the map sends a heteromorphism $(a,f):\Hom_{\cat{S}_\cat{C}}(p,q)$ to the heteromorphism $(a, F(f)):\Hom_{\cat{S}_\cat{C'}}(p,q)$. We leave the remaining equations to the reader.
\end{proof}

\begin{corollary}\label{cor.sel_functorial}
The map $\lens{p\circ-}{p}\colon\smcat\to\smcat$ from \cref{prop.selection} is functorial.
\end{corollary}
\begin{proof}
Given $F\colon\cat{C}\to\cat{C'}$, the proof of \cref{thm.sel_enriched} defined a functor $\cat{S}_F(p)\colon\lens{p\circ\cat{C}}{p}\to\lens{p\circ\cat{C'}}{p}$. The respecting of identities and compositions are straightforward.
\end{proof}

Let $\pprof_{\bo}$ (resp.\ $\pprof_{\ff}$) denote the wide, locally full subcategory of $\pprof$ whose tight morphisms are bijective-on-objects (resp.\ fully faithful) functors. Accordingly, let $\pprofcat[\bo]$ (resp.\ $\pprofcat[\ff]$) denote the category of $\pprof_\bo$-enriched (resp.\ $\pprof_\ff$-enriched) categories and functors.

\begin{proposition}\label{prop.selection_boff}
The functor $\cat{S}_-$ from \cref{thm.sel_enriched} restricts to functors
\[
	\cat{S}_-\colon\smcat_{\bo}\too\pprofcat[\bo]
	\qqand
	\cat{S}_-\colon\smcat_{\ff}\too\pprofcat[\ff],
\]
meaning that if $F\colon\cat{C}\to\cat{C'}$ is bijective on objects (resp.\ fully faithful) then so is the functor $\cat{S}_F(p)$ for each $p:\poly$.
\end{proposition}
\begin{proof}
Investigating the proof of \cref{thm.sel_enriched}, one sees that $\cat{S}_F(p)$ is $\bo$ (resp.\ $\ff$) if $F$ is.
\end{proof}

While it is beyond the scope of this paper, we conjecture that the considerations in this section arise because $\lens{p\circ-}{p}$ sends left adjoint bicomodules to left adjoint bicomodules, and preserves the ternary factorization system $(\Delta_\ff\then\Delta_\bo\then\Sigma_\dopf)$ of left adjoint bicomodules.

\begingroup
\linespread{1}\selectfont
\printbibliography
\endgroup

\end{document}